\documentclass{jcmlatex}
\usepackage{amsmath,multirow}
\usepackage{amssymb}
\usepackage{amscd}
\usepackage{amsfonts}
\usepackage{enumerate}
\usepackage{mathrsfs}
\usepackage{xypic}
\usepackage{stmaryrd}
\usepackage{graphicx}
\usepackage{fancybox}
\usepackage{portland}
\usepackage{color}
\usepackage{amsmath}
\usepackage{exscale}
\usepackage{enumitem,array}%
\usepackage{amssymb}
\usepackage{graphicx}

\newcommand{\norm}[1]{\left\Vert#1\right\Vert}
\def\JCMvol{xx}
\def\JCMno{x}
\def\JCMyear{200x}
\def\JCMreceived{xxx}
\def\JCMrevised{xxx}
\def\JCMaccepted{xxx}

\begin{document}

\markboth{B.~WANG, X.~WU, F.~MENG AND Y.~FANG}{Exponential Fourier
collocation methods}

\title{EXPONENTIAL FOURIER COLLOCATION METHODS FOR SOLVING FIRST-ORDER DIFFERENTIAL EQUATIONS}

\author{Bin Wang
\thanks{School of Mathematical Sciences, Qufu Normal
University, Qufu 273165, P.R.China \\ Email: wangbinmaths@gmail.com}
\and Xinyuan Wu
\thanks{Department of
Mathematics, Nanjing University, Nanjing  210093, P.R.China\\
Email: xywu@nju.edu.cn} \and Fanwei Meng
\thanks{School of Mathematical Sciences, Qufu Normal
University, Qufu 273165, P.R.China\\
Email: fwmeng@qfnu.edu.cn} \and Yonglei Fang
\thanks{School of Mathematics  and Statistics, Zaozhuang
University,  Zaozhuang 277160, P.R. China\\
Email: ylfangmath@163.com}}

\maketitle

\begin{abstract}
In this paper, a novel class   of exponential Fourier collocation
methods (EFCMs) is presented
 for solving systems of first-order  ordinary   differential equations. These so-called exponential  Fourier collocation
methods   are based on the variation-of-constants formula,
incorporating a local Fourier expansion of the underlying problem
with collocation methods. We discuss  in detail  the connections of
EFCMs with trigonometric Fourier collocation methods (TFCMs), the
well-known Hamiltonian Boundary Value Methods (HBVMs), Gauss methods
and Radau IIA methods. It  turns out  that the novel EFCMs are an
essential extension of
 these existing methods.   We also analyse the accuracy in preserving the
quadratic invariants and the Hamiltonian energy when the underlying
system is a Hamiltonian system.  Other properties of EFCMs including
the order of approximations and the convergence of fixed-point
iterations are investigated  as well.  The analysis given in this
paper proves further that EFCMs can achieve arbitrarily high order
in a  routine manner which
 allows us to construct higher-order methods for solving systems of
first-order
 ordinary  differential equations conveniently.  We   also derive  a
practical fourth-order EFCM denoted by EFCM(2,2) as  an illustrative
example. The numerical experiments  using
  EFCM(2,2)  are  implemented  in comparison with an existing
fourth-order HBVM, an energy-preserving collocation method and a
fourth-order exponential integrator in the literature. The numerical
results demonstrate the remarkable efficiency  and robustness of the
novel EFCM(2,2).
\end{abstract}



\section{Introduction}This paper is devoted to analysing and designing novel and efficient numerical integrators
for solving the following first-order   initial value problems
\begin{equation}
u^{\prime}(t)+Au(t)=g(t,u(t)),  \qquad
u(0)=u_0,\qquad t\in[0,t_{\mathrm{end}}],\label{prob}%
\end{equation}
where $g: \mathbb{R} \times\mathbb{R}^{d}\rightarrow \mathbb{R}^{d}$
is an analytic function,
 $A$ is assumed to be a  linear
operator on a Banach space $X$ with a norm $\norm{\cdot}$, and
$(-A)$ is the infinitesimal generator of a strongly continuous
semigroup $e^{-tA}$ on $X$ (see, e.g. \cite{Hochbruck2010}). This
assumption of $A$  means  that    there exist  two  constants $C$
and $\omega$  satisfying
 \begin{equation}\norm{e^{-tA}}_{X\leftarrow X}\leq Ce^{\omega t},\ \ \ \ \ t\geq0.\label{condition A}%
\end{equation}
An analysis about this result can be found in \cite{Hochbruck2010}.
It is noted that if $X$ is chosen as $X=\mathbb{R}^{d}$ or
$X=\mathbb{C}^{d}$, then the linear operator $A$ can be  expressed
by a $d\times d$ matrix.  Accordingly in this case, $e^{-tA}$  is
exactly the matrix exponential function. It also can be observed
that the condition  \eqref{condition A} holds with $\omega=0$
 provided the field of values of $A$ is contained in the right
complex half-plane.  In the special and important case where  $A$ is
skew-Hermitian or Hermitian positive semidefinite,  we have  $C = 1$
and $\omega=0$ in the Euclidean norm, independently of the dimension
$d$. If $A$  originates  from a spatial  discretisation of a partial
differential equation, then the assumption of $A$  leads to temporal
convergence results that are independent of the spatial mesh.

 It is known that  the exact solution of \eqref{prob} can be
represented by the variation-of-constants formula
  \begin{equation}
 u(t)= e^{-tA}u_0+  \int_{0}^t e^{-(t-\tau)A}g(\tau,u(\tau))d\tau.\\
\label{probsystemF3}%
\end{equation}
 For oscillatory problems, the exponential
subsumes the full information on linear oscillations.
 This class of  problems
\eqref{prob} frequently rises in a wide variety of applications
including engineering, mechanics, quantum physics, circuit
simulations, flexible body dynamics and other applied sciences (see,
e.g.
\cite{Brugnano2011-stiff,Grimm2006,Hochbruck1998,Hochbruck2010,wang-2014DCD,wang2015,wu2014-JCP,wu-book}).
 Parabolic partial
differential equations with their spatial discretisations and highly
oscillatory problems are two typical examples of the system
\eqref{prob} (see, e.g.
\cite{Iserles2002a,Iserles2002b,Kassam2005,Khanamiryan2008,Krogstad2003,wu2013-JCP}).
Linearizing   stiff systems $u^{\prime}(t) =F(t,u(t))$ also  yields
examples of the form \eqref{prob}(see, e.g.
\cite{Cox2002,Hochbruck2005,Hochbruck2009}).

Based on the variation-of-constants formula \eqref{probsystemF3},
the numerical scheme for \eqref{prob} is usually constructed by
incorporating the exact propagator of \eqref{prob} in an appropriate
way. For example, interpolating the nonlinearity at the known value
$g(0,u_0)$ yields the exponential Euler approximation for
\eqref{probsystemF3}. Approximating the functions arising by
rational approximations leads to implicit or semi-implicit
Runge--Kutta methods, Rosenbrock methods or W-schemes. Recently, the
construction, analysis, implementation and application of
exponential integrators have been studied by many researchers, and
we refer the
 reader to \cite{Berland-2005,Caliari-2009,Calvo-2006,Celledoni-2008,Grimm2006,Ostermann2006,wu2012-6},
for example. Exponential integrators make explicit use of the
quantity $Au$ of \eqref{prob},  and a systematic survey of
exponential integrators is referred to \cite{Hochbruck2010}.

Based on Lagrange interpolation polynomials, exponential Runge-Kutta
methods of collocation type are constructed and their convergence
properties are analysed in \cite{Hochbruck2005-1}. In
\cite{wang-2014}, the authors developed and researched a novel type
of trigonometric Fourier collocation methods (TFCMs) for
second-order oscillatory differential equations
$q^{\prime\prime}(t)+Mq(t)=f(q(t))$ with a principal frequency
matrix $M\in\mathbb{R}^{d\times d}$. These new trigonometric Fourier
collocation methods take full advantage of the special structure
brought by the linear term $Mq$, and its construction incorporates
the idea of collocation methods, the variation-of-constants formula
and the local Fourier expansion of the system.   The results of
numerical experiments in \cite{wang-2014} showed  that the
trigonometric
 Fourier collocation methods are
much more efficient in comparison with some alternative approaches
that have previously appeared in the literature. On the basis of the
{work in \cite{Hochbruck2005-1,wang-2014},} in this paper we make an
effort to conduct the research of novel exponential
 Fourier collocation methods (EFCMs)  for efficiently solving first-order differential
 equations \eqref{prob}. The construction  of the novel  EFCMs  incorporates
the exponential integrators, the collocation methods, and the local
Fourier expansion of the system. Moveover, EFCMs
 can be of an arbitrarily high order, and when $A\rightarrow0$,   EFCMs
reduce to  the well-known Hamiltonian Boundary Value methods (HBVMs)
which have been studied by many researchers (see, e.g.
\cite{Brugnano2010,Brugnano2011,Brugnano2012}). It is also shown in
this paper that  EFCMs are an extension of Gauss methods, Radau IIA
methods  and TFCMs.

The paper is organized as follows. We first formulate the scheme of
EFCMs in Section \ref{sec:Computations}. Section
\ref{sec:Connections} discusses the connections of the novel EFCMs
with HBVMs, Gauss methods, Radau IIA methods   and TFCMs.  In
Section \ref{sec:Analysis of the methods}, we  analyse the
properties of EFCMs.  Section \ref{numerical experiments}  is
concerned with constructing a practical EFCM and reporting four
numerical experiments to demonstrate the excellent qualitative
behavior of the novel approximation. Section \ref{sec:conclusions}
includes some
 conclusions.

\section{Formulation of EFCMs} \label{sec:Computations}
In this section, we  present the formulation of exponential Fourier
collocation methods (EFCMs) for systems of first-order differential
equations \eqref{prob}.

\subsection{Local Fourier expansion} \label{sec: Local
Fourier expansion} We first restrict the first-order differential
equations \eqref{prob}
 to the interval $[0,h]$ with any $h>0$:
\begin{equation}
u'(t)+Au(t)=g(t,u(t)),  \qquad
u(0)=u_0,\qquad t\in[0,h].\label{prob h}%
\end{equation}
Consider  the shifted Legendre polynomials
$\{\widehat{P}_j\}_{j=0}^{\infty}$  satisfying
$$\int_{0}^{1}\widehat{P}_i(x)\widehat{P}_j(x)dx=\delta_{ij},\qquad \deg  \big(\widehat{P}_j\big)=j,\qquad i,j\geq0,$$
where $\delta_{ij}$ is the Kronecker symbol.  We then expand the
right-hand-side function of \eqref{prob h} as follows:
\begin{equation}g(\xi h,u(\xi h))=\sum\limits_{j=0}^
{\infty}\widehat{P}_j(\xi )\kappa_j(h,u), \ \ \xi \in[0,1]; \ \
\kappa_j(h,u):=\int_{0}^{1}\widehat{P}_j(\tau)g(\tau h,u(\tau h))d\tau.\label{fq}%
\end{equation}
The system \eqref{prob h} now can be rewritten as
\begin{equation}
u'(\xi h)+Au(\xi h)=\sum\limits_{j=0}^ {\infty}\widehat{P}_j(\xi
)\kappa_j(h,u),  \qquad
u(0)=u_0.\label{prob series h}%
\end{equation}
The next theorem gives its  solution.
\begin{theorem}
\label{vcf expansion}  The solution of \eqref{prob h} can be
expressed by
\begin{equation}
\begin{aligned} u(t) =\varphi_{0}(-tA)u_0+t\sum\limits_{j=0}^ {\infty}I_{j}(tA)\kappa_j(t,u),
\end{aligned}
\label{expansion Solution}%
\end{equation}
where $t\in[0,h]$ and \begin{equation}
\begin{aligned} &I_{j}(tA):=\int_{0}^1
\widehat{P}_j(z)e^{-(1-z)tA}dz =\sqrt{2j+1}\sum\limits_{k=0}^
{j}(-1)^{j+k}\frac{(j+k)!}{k!(j-k)!}\varphi_{k+1}(-tA).
\end{aligned}
\label{Ij}%
\end{equation}
Here  the $\varphi$-functions (see, e.g.
\cite{Hochbruck1998,Hochbruck2005,Hochbruck2010,Hochbruck2009}) are
defined by:
\begin{equation*}
\varphi_0(z)=e^{z},\ \ \varphi_k(z)=\int_{0}^1
e^{(1-\sigma)z}\frac{\sigma^{k-1}}{(k-1)!}d\sigma, \ \ k=1,2,\ldots.
\label{phi}%
\end{equation*}
\end{theorem}
\begin{proof}
It follows from  the variation-of-constants formula
\eqref{probsystemF3} that
\begin{equation*}
\begin{aligned} u(t)&=e^{-tA}u_0+  \int_{0}^t e^{-(t-\tau)A}g(\tau,u(\tau))d\tau\\
&=\varphi_{0}(-tA)u_0+ t \int_{0}^1 e^{-(1-z)tA}g(zt,u(zt))dz.
\end{aligned}
\label{prob expansion Solution}%
\end{equation*}
Replacing  the function $g(zt,u(zt))$ in the integral by \eqref{fq}
yields
\begin{equation*}
\begin{aligned} u(t)
&=\varphi_{0}(-tA)u_0+ t \int_{0}^1 e^{-(1-z)tA}\sum\limits_{j=0}^
{\infty}\widehat{P}_j(z)\kappa_j(t,u)dz\\
&=\varphi_{0}(-tA)u_0+ t \sum\limits_{j=0}^
{\infty}\int_{0}^1\widehat{P}_j(z) e^{-(1-z)tA}dz\kappa_j(t,u),
\end{aligned}%
\end{equation*}
which gives the formula \eqref{expansion Solution}  by letting
$I_{j}(tA)=\int_{0}^1 \widehat{P}_j(z)e^{-(1-z)tA}dz$.

According to the definition of   shifted Legendre polynomials in the
interval $[0,1]$:
\begin{equation}
\widehat{P}_j(x)=(-1)^j\sqrt{2j+1}\sum\limits_{k=0}^ {j}{j
\choose{k}}{j+k \choose{k}}(-x)^k,\qquad j=0,1,\ldots,\qquad
x\in[0,1],
\label{Pj}%
\end{equation}
we arrive at \begin{equation*}\label{Ijta}
\begin{aligned} &I_{j}(tA)
=\int_{0}^1
\widehat{P}_j(z)e^{-(1-z)tA}dz\\
=&\int_{0}^1 (-1)^j\sqrt{2j+1}\sum\limits_{k=0}^ {j}{j
\choose{k}}{j+k \choose{k}}(-z)^k e^{-(1-z)tA}dz\\
=&\sqrt{2j+1}\sum\limits_{k=0}^ {j}(-1)^{j+k}{j \choose{k}}{j+k
\choose{k}}\int_{0}^1 z^k e^{-(1-z)tA}dz\\
=&\sqrt{2j+1}\sum\limits_{k=0}^
{j}(-1)^{j+k}\frac{(j+k)!}{k!(j-k)!}\varphi_{k+1}(-tA).
\end{aligned}%
\end{equation*}
\end{proof}

\subsection{Discretisation} \label{subsec:Discretization}
The authors in \cite{Brugnano2012} made use of interpolation
quadrature formulae and gave the discretisation   for initial value
problems. Following \cite{Brugnano2012},   two tools are coupled in
this part. We first   truncate the local Fourier expansion
  after a finite number of
terms and then  compute the coefficients of the expansion   by a
suitable quadrature formula.

We  now  consider truncating the Fourier expansion, a technique
which originally appeared in
  \cite{Brugnano2012}.   This can be achieved by
truncating the series \eqref{expansion Solution} after $n$
($n\geq2$) terms with the stepsize $h$ and $V:=hA$:
\begin{equation}
\begin{aligned} &\tilde{u}(h)=\varphi_{0}(-V)u_0+h\sum\limits_{j=0}^
{n-1}I_{j}(V)\kappa_j(h,\tilde{u}),
\end{aligned}
\label{basic method}%
\end{equation}
which satisfies the following initial value problem:
\begin{equation*}
\begin{aligned}
 & \tilde{u}'(\xi h)+A\tilde{u}(\xi h)=\sum\limits_{j=0}^
{n-1}\widehat{P}_j(\xi )\kappa_j(h,\tilde{u}),\ \ \ \
\tilde{u}(0)=u_{0}.\end{aligned}
\label{truncating H-s}%
\end{equation*}

The key challenge in designing practical methods  is  how to deal
with $\kappa_j(h,\tilde{u})$ effectively. To this end, we introduce
a quadrature formula using $k$   abscissae $0\leq c_1\leq\ldots\leq
c_k\leq1$ {and being exact for polynomials of degree up to $m-1$}.
It is required that $m\geq k$ in this paper, and we note that many
existed quadrature formulae satisfy this requirement, such as the
well-known Gauss--Legendre
 quadrature  and the Radau quadrature.  We  thus  obtain an approximation of the form
\begin{equation}
\begin{aligned} \kappa_j(h,\tilde{u})\approx \sum\limits_{l=1}^
{k}b_l\widehat{P}_j(c_l)g(c_l h,\tilde{u}(c_l h)),\ \ \ \
j=0,1,\ldots, n-1,
\end{aligned}
\label{Computate gamma}%
\end{equation} where $b_l$  for $l=1,2,\ldots,k$  are the
quadrature weights.  It is noted that since the number of the
integrals  $\kappa_j(h,\tilde{u})$ is $n$,  it is assumed that
$k\geq n$. Therefore, we  have  $m\geq n$.

Since the quadrature is exact for polynomials of degree $m-1$, its
remainder depends on the $m$-th derivative of the integrand
$\widehat{P}_j(\tau)g(\tau h,u(\tau h))$ with respect to $\tau$.
Consequently,  the approximation gives
\begin{equation*}
\begin{aligned}
&\Delta_j(h,\tilde{u}):= \kappa_j(h,\tilde{u})-\sum\limits_{l=1}^
{k}b_l\widehat{P}_j(c_l)g(c_lh,\tilde{u}(c_l h))\\
=&\int_{0}^{1}\widehat{P}_j(\tau)g(\tau h,u(\tau
h))d\tau-\sum\limits_{l=1}^
{k}b_l\widehat{P}_j(c_l)g(c_lh,\tilde{u}(c_l h))\\
=&C\int_{0}^{1}\dfrac{d^{m}\Big(\widehat{P}_j(\tau)g(\tau h,u(\tau
h))\Big)}{d\tau^{m} }|_{\tau=\zeta}d\tau,
\end{aligned}
\end{equation*}
where $C$ is a constant, and $\zeta\ (\zeta\in[0,1])$ depends on
$\tau$. Taking account of $\widehat{P}^{(k)}_j(\tau)=0$ for $k>j$,
we obtain
\begin{equation*}
\begin{aligned}
\Delta_j(h,\tilde{u})=&C\int_{0}^{1}\widehat{P}_j(\zeta)\hat{g}^{(m)}(\zeta
h)d\tau h^m+Cm\int_{0}^{1}\widehat{P}'_j(\zeta)\hat{g}^{(m-1)}(\zeta
h)d\tau h^{m-1}\\
&+\cdots+C{m
\choose{j}}\int_{0}^{1}\widehat{P}^{(j)}_j(\zeta)\hat{g}^{(m-j)}(\zeta
h)d\tau h^{m-j}=\mathcal{O}(h^{m-j}),\\
&\qquad \qquad \qquad \qquad \qquad \qquad \qquad \qquad \qquad
\qquad j=0,1,\ldots,n-1
\end{aligned}
\end{equation*}
with the notation $\hat{g}^{(k)}(\zeta h)=   g^{(k)}(\zeta h,u(\zeta
h)) .$
 This  guarantees that  each $\Delta_j(h,\tilde{u})$ has good accuracy for any $j=0,1,\ldots,n-1$. Choosing $k$ large enough, along with a
suitable choice of   $c_l,\ b_l$ for $\ l=1,2,\ldots,k$,  allows us
to approximate the given integral $\kappa_j(h,\tilde{u})$ to any
degree of accuracy.

With \eqref{basic method} and  \eqref{Computate gamma}, it is
natural to consider the following numerical scheme
\begin{equation*}
\begin{aligned} &v(h)=\varphi_{0}(-V)u_0+h\sum\limits_{j=0}^
{n-1}I_{j}(V)\sum\limits_{l=1}^ {k}b_l\widehat{P}_j(c_l)g(c_l
h,v(c_l h)),
\end{aligned}
\label{quadrature basic method}%
\end{equation*}
which exactly solves  the  initial value
 problem as follows:
\begin{equation}
\begin{aligned}
 & v'(\xi h)=-Av(\xi h)+\sum\limits_{j=0}^
{n-1}\widehat{P}_j(\xi )\sum\limits_{l=1}^
{k}b_l\widehat{P}_j(c_l)g(c_l h,v(c_l h)),\ \ \ \  v(0)=u_{0}.
\end{aligned}
\label{quadrature truncating H-s}%
\end{equation}
It  follows from \eqref{quadrature truncating H-s}  that $v(c_i h)$
for $ i=1,2,\ldots,k$ satisfy the following first-order differential
equations:
\begin{equation}
\begin{aligned}
&v'(c_i h)+Av(c_i h)=\sum\limits_{j=0}^
{n-1}\widehat{P}_j(c_i)\sum\limits_{l=1}^
{k}b_l\widehat{P}_j(c_l)g(c_l h,v(c_l h)),\ \ \  v(0)=u_{0}.\\
\end{aligned}
\label{discrete quadrature truncating H-s}%
\end{equation}
Letting $v_i=v(c_i h),$ \eqref{discrete quadrature truncating H-s}
can be solved by the variation-of-constants formula
\eqref{probsystemF3}  of the form:
\begin{equation*}
\begin{aligned}
v_i=&\varphi_{0}(-c_iV)u_0+c_ih\sum\limits_{j=0}^ {n-1}I_{j,c_i}(V)
\sum\limits_{l=1}^ {k}b_l\widehat{P}_j(c_l)g(c_l h,v_l),\ \ \
i=1,2,\ldots,k,
\end{aligned}
\label{discrete quadrature method}%
\end{equation*}
 where
 {\begin{equation}
\begin{aligned} &I_{j,c_i}(V):=\int_{0}^1\widehat{P}_j(c_iz)e^{-(1
-z)c_iV}dz\\
=&\int_{0}^1 (-1)^j\sqrt{2j+1}\sum\limits_{k=0}^ {j}{j
\choose{k}}{j+k \choose{k}}(-c_iz)^k e^{-(1-z)c_iV}dz\\
=&(-1)^j\sqrt{2j+1}\sum\limits_{k=0}^ {j}(-c_i)^k{j \choose{k}}{j+k
\choose{k}}\int_{0}^1 z^k e^{-(1-z)c_iV}dz\\
=&(-1)^j\sqrt{2j+1}\sum\limits_{k=0}^
{j}(-c_i)^k\frac{(j+k)!}{k!(j-k)!}\varphi_{k+1}(-c_iV).
\end{aligned}
\label{Ijci}%
\end{equation}}
\subsection{The exponential  Fourier collocation methods} \label{subsec:The methods}
 We are now in a position to present the novel  exponential  Fourier collocation methods
for systems of  first-order   differential equations \eqref{prob}.
\begin{definition}
\label{numerical method}  The $k$-stage  exponential  Fourier
collocation
 method with an integer $n$ (denoted by EFCM(k,n))
 for integrating systems of   first-order differential
 equations \eqref{prob}   is
defined by
\begin{equation}
\begin{aligned} v_i=&\varphi_{0}(-c_iV)u_0+c_ih\sum\limits_{l=1}^ {k}b_l\Big(\sum\limits_{j=0}^ {n-1}I_{j,c_i}(V)
\widehat{P}_j(c_l)\Big)g(c_l h,v_l),\
i=1,2,\ldots,k,\\
v(h)&=\varphi_{0}(-V)u_0+h\sum\limits_{l=1}^
{k}b_l\Big(\sum\limits_{j=0}^
{n-1}I_{j}(V)\widehat{P}_j(c_l)\Big)g(c_l h,v _l),
\end{aligned}
\label{methods}%
\end{equation}
where $h$ is the stepsize, $V:=hA$, $\widehat{P}_j$  for
$j=0,1,\ldots,n-1$ are defined by \eqref{Pj},  and $c_l,\ b_l$ for
$l=1,2,\ldots,k$  are the node points and the quadrature weights of
a quadrature formula, respectively. Here, $n$ is an integer which is
required to satisfy the condition: $2\leq n\leq k$. { $I_{j}(V)$ and
$I_{j,c_i}(V)$ are  determined  by
\begin{equation*}
\begin{aligned} &I_{j}(V) =\sqrt{2j+1}\sum\limits_{k=0}^
{j}(-1)^{j+k}\frac{(j+k)!}{k!(j-k)!}\varphi_{k+1}(-V),\\
&I_{j,c_i}(V) =(-1)^j\sqrt{2j+1}\sum\limits_{k=0}^
{j}(-c_i)^k\frac{(j+k)!}{k!(j-k)!}\varphi_{k+1}(-c_iV).
\end{aligned}
\end{equation*}}
\end{definition}
\begin{remark}
Clearly, it can be observed that  the EFCM(k,n) defined by
\eqref{methods} exactly integrates  the homogeneous linear system
$u'+Au=0$, thus it is trivially  A-stable.
 The EFCM(k,n)    \eqref{methods} approximates the solution of \eqref{prob} in the time interval $[0,h]$. Obviously,
the obtained result $v(h)$  can be considered as  the initial
condition for a new initial value problem and $u(t)$ can be
approximated in the time interval $[h,2h]$. In general, the
EFCM(k,n) can be extended to the approximation of the solution in an
arbitrary interval $[0,Nh]$,  where $N$ is a positive integer.
\end{remark}

\begin{remark}
The novel EFCM(k,n) \eqref{methods} developed here  is a kind of
 exponential integrator which requires  the
approximation of products of  $\varphi$-functions with vectors. It
is noted that if $A$ has a simple structure, it is possible to
compute the $\varphi$-functions in a fast and reliable way.
Moveover, many different approaches to evaluating this action in an
efficient way have been proposed in the literature, see, e.g.
\cite{Higham2009,Higham2011,Berland2007,Higham2010acta,Hochbruck1997,Hochbruck2010,Lubich-book,Moler-2003}.
Furthermore,  all the matrix functions appearing in the EFCM(k,n)
\eqref{methods}  only need to be calculated once in the actual
implementation for the given stepsize $h$. In Section \ref{numerical
experiments}, we will compare our novel methods with some
traditional collocation methods (which do not require the evaluation
of matrix functions) by four experiments. For each problem, we will
display the work precision diagram in which the global error is
plotted versus the execution time. The numerical results given in
Section \ref{numerical experiments} demonstrate the efficiency of
our novel approximation.

\end{remark}

\section{Connections with some existing methods}
\label{sec:Connections}

 So far  various effective methods have been developed for solving
first-order differential equations and this section is devoted to
exploring the connections between our novel  EFCMs and some other
existing methods in the literature.  It turns out that some existing
traditional methods can be gained by letting $A\rightarrow 0$ in the
corresponding EFCMs or by applying EFCMs to special second-order
differential equations.

\subsection{Connections with HBVMs and Gauss methods}

 Hamiltonian Boundary Value methods (HBVMs) are an interesting
class of integrators, which exactly preserve energy of polynomial
Hamiltonian systems (see, e.g.
\cite{Brugnano2010,Brugnano2011,Brugnano2012}). We first consider
the  connection between EFCMs and HBVMs.

 It can be observed that from \eqref{Ijci} that when  $A\rightarrow
0$, $I_{j}(V)$ and $I_{j,c_i}(V)$ in \eqref{methods} become
\begin{equation*}
\begin{aligned}
\tilde{I}_{j} :&=I_{j}(0)=\int_{0}^1
 \widehat{P}_j(z)dz=\left\{
\begin{aligned}
 &1,\ j=0,\\
 &0, \ j\geq1,\end{aligned}\right.\\
 \tilde{I}_{j,c_i}
:&=I_{j,c_i}(0)=\int_{0}^1
 \widehat{P}_j(c_iz)dz.\\
\end{aligned}
\end{equation*}
This can be summed up in the following result.
\begin{theorem}\label{HBVM thm}
When $A\rightarrow 0$, the EFCM(k,n) defined by \eqref{methods}
reduces to
\begin{equation}
\begin{aligned} v_i=&u_0+c_ih\sum\limits_{l=1}^ {k}b_l\Big(\sum\limits_{j=0}^ {n-1} \tilde{I}_{j,c_i}
\widehat{P}_j(c_l)\Big)g(c_l h,v_l),\
i=1,2,\ldots,k,\\
v(h)&=u_0+h \sum\limits_{l=1}^ {k}b_l g(c_l h,v _l),
\end{aligned}
\label{methods0}%
\end{equation}
which can be rewritten as a $k$-stage Runge-Kutta method with the
following Butcher tableau
\begin{equation}\label{Radau datas}
\begin{tabular}
[c]{c|ccc}%
$c_{1}$ &  \\
$\vdots$ & $ $ &  $\bar{A}=(\bar{a}_{ij})_{k\times
k}=\Big(b_{j}\sum\limits_{l=0}^ {n-1}
\widehat{P}_l(c_j)\int_{0}^{c_i}
 \widehat{P}_l( \tau)d\tau\Big)_{k\times k}$  \\
$c_{k}$ &  \\\hline & $\raisebox{-1.3ex}[1.0pt]{$\ b_1$}$ &
$\raisebox{-1.3ex}[1.0pt]{$\cdots$}$ &
$\raisebox{-1.3ex}[1.0pt]{$b_k$}$%
\end{tabular}
\end{equation}
This method is exactly the  Hamiltonian Boundary Value Method
HBVM(k,n) using  the discretisation researched in
\cite{Brugnano2010,Brugnano2011,Brugnano2012} for the first-order
system $$u^{\prime}(t)=g(t,u(t)), \ u(0)=u_0.$$
\end{theorem}

From the property of HBVM(k,n) given in \cite{Brugnano2012}, it
follows that  HBVM(k,k) reduces to a $k$-stage Gauss-Legendre
collocation method when  a Gaussian distribution of the nodes
$(c_1,\cdots,c_k)$ is used.  In view of this   and as an
straightforward consequence of Theorem \ref{HBVM thm}, we obtain the
connection between EFCMs and Gauss methods.  This result is
described below.

\begin{theorem}\label{invariant RKN thm}
Under the condition that  $c_l,\ b_l$ for $ l=1,2,\ldots,k$  are
chosen respectively as the node points and the quadrature weights of
a $k$-point Gauss--Legendre
 quadrature over the interval $[0,1]$,  then the EFCM(k,k) defined by \eqref{methods}
 reduces to the $k$-stage Gauss method presented in \cite{hairer2006} when   $A\rightarrow
 0$.
\end{theorem}

\subsection{Connection between EFCMs and Radau IIA methods}

The following  theorem  states the connection between EFCMs and
Radau IIA methods.

 \begin{theorem}\label{Radau thm}
Choose $c_l,\ b_l$ for $l=1,2,\ldots,k$   respectively as  the node
points and the weights of the Radau-right quadrature formula. Then
the EFCM(k,k) defined by \eqref{methods}
 reduces to a $k$-stage  Radau IIA method  presented in \cite{hairer-stiff} when  $A\rightarrow 0$.
\end{theorem}
\begin{proof}
It follows from Theorem \ref{HBVM thm} that when $A\rightarrow 0,$
the EFCM(k,k) defined by \eqref{methods} reduces to \eqref{methods0}
with $n=k$.  According to \cite{hairer-stiff}, the shifted Legendre
polynomials $\{\widehat{P}_j\}_{j=0}^{\infty}$ satisfy the
 following  integration formulae
\begin{equation*}
\begin{aligned}
&\int_{0}^{x}\widehat{P}_0( t)dt=\xi_1\widehat{P}_1( x)+\frac{1}{2}\widehat{P}_0( x),\\
 &\int_{0}^{x}\widehat{P}_m( t)dt=\xi_{m+1}\widehat{P}_{m+1}( x)-\xi_{m}\widehat{P}_{m-1}( x),\ m=1,2,\ldots,k-2,\\
 & \int_{0}^{x}\widehat{P}_{k-1}( t)dt= \beta_{k}\widehat{P}_{k-1}( x)-\xi_{k-1}\widehat{P}_{k-2}( x),\end{aligned}
\end{equation*}
where $$\xi_{m}=\frac{1}{2\sqrt{4m^2-1}},\ \ \
\beta_{k}=\frac{1}{4k-2}.$$ These formulae imply
\begin{equation*}
\begin{aligned}
\bar{A} =&\left(
          \begin{array}{ccc}
            \int_{0}^{c_1}\widehat{P}_0( \tau)d\tau &  \ldots & \int_{0}^{c_1}\widehat{P}_{k-1}( \tau)d\tau \\
            \vdots &   & \vdots \\
           \int_{0}^{c_k}\widehat{P}_0( \tau)d\tau &  \ldots & \int_{0}^{c_k}\widehat{P}_{k-1}( \tau)d\tau \\
          \end{array}
        \right)
        \left(
          \begin{array}{ccc}
            b_1\widehat{P}_0( c_1) &  \ldots &  b_s\widehat{P}_0( c_s)\\
            \vdots &   & \vdots \\
            b_1\widehat{P}_{k-1}( c_1) &  \ldots &  b_s\widehat{P}_{k-1}( c_s)\\
          \end{array}
        \right)\\
        =&W X_k Q
,\end{aligned}
\end{equation*}
where the matrix  $W$ is defined by
$$\omega_{ij}=\widehat{P}_{j-1}( c_i),\ \ \ \ i,j=1,\ldots,k,$$
and the matrices $X_k,\ Q$ are  determined  by
\begin{equation}\label{Xk Q}
\begin{aligned}
& X_k=\left(
                 \begin{array}{ccccc}
                   \frac{1}{2} &-\xi_1 &   &   &   \\
                   \xi_1 &0 & -\xi_2 &   &   \\
                     & \ddots & \ddots & \ddots &   \\
                    &   & \xi_{k-2}& 0 & -\xi_{k-1} \\
                     &   &   & \xi_{k-1} & \beta_k \\
                 \end{array}
               \right),\ \ Q=\left(
          \begin{array}{ccc}
            b_1\widehat{P}_0( c_1) &  \ldots &  b_s\widehat{P}_0( c_s)\\
            \vdots &   & \vdots \\
            b_1\widehat{P}_{k-1}( c_1) &  \ldots &  b_s\widehat{P}_{k-1}( c_s)\\
          \end{array}
        \right).\end{aligned}
\end{equation}

Based on  the fact that  the Radau-right quadrature formula is of
order $2k-1$, we  obtain that polynomials $\widehat{P}_{m}(x)
\widehat{P}_{n}(x)\ (m+n\leq2k-2)$ are integrated exactly by this
quadrature formula, i.e.,
$$\sum\limits_{i=1}^ {k}b_i\widehat{P}_{m}(c_i)
\widehat{P}_{n}(c_i)=\int_{0}^{1}\widehat{P}_{m}(x)
\widehat{P}_{n}(x)dx=\delta_{mn},$$ which  means  $WQ=I.$
 Therefore,
$$\bar{A} =W X_k W^{-1}.$$
 \eqref{Radau datas} now becomes \begin{equation*}
\begin{tabular}
[c]{c|ccc}%
$c_{1}$ &  \\
$\vdots$ & $ $ &  $\bar{A}=W X_k W^{-1}$  \\
$c_{k}$ &  \\\hline & $\raisebox{-1.3ex}[1.0pt]{$\ b_1$}$ &
$\raisebox{-1.3ex}[1.0pt]{$\cdots$}$ &
$\raisebox{-1.3ex}[1.0pt]{$b_k$}$%
\end{tabular}
\end{equation*}
which is exactly the same as the scheme of Radau IIA method
presented in \cite{Brugnano2015} by using the  W-transformation.
\end{proof}

\subsection{Connection between EFCMs and TFCMs}

 A novel type of trigonometric Fourier collocation methods (TFCMs)
for second-order oscillatory differential equations
\begin{equation}
q^{\prime\prime}(t)+Mq(t)=f(q(t)),  \qquad
q(0)=q_0,\ \ q'(0)=q'_0\label{old-prob}%
\end{equation}
has been  developed and researched in \cite{wang-2014}. These
methods can attain arbitrary algebraic order in a very simple way,
which is very important for solving systems of second-order
oscillatory ODEs. This part  is devoted to clarifying the connection
between EFCMs and TFCMs.

We apply the TFCMs  presented in \cite{wang-2014} to
\eqref{old-prob} and denote the numerical solution by $(v_{T},\
u_{T})^\intercal$. According to the analysis in \cite{wang-2014}, it
is known that the numerical solution satisfies the following
differential equation
\begin{equation}\begin{aligned}& \left(
                   \begin{array}{c}
                     v_{T}(\xi h) \\
                      u_{T}(\xi h) \\
                   \end{array}
                 \right)
'=\left(
                                                                           \begin{array}{c}
                                                                             u_{T}(\xi h) \\
                                                                            -Mv_{T}(\xi h)+\sum\limits_{j=0}^ {n-1}\widehat{P}_j(\xi )\sum\limits_{l=1}^ {k}b_l\widehat{P}_j(c_l)f(v_{T}(c_l h))
                                                                           \end{array}
                                                                         \right)
\label{TFCM equation}%
\end{aligned}\end{equation}
with the initial value $$\big( v_{T}(0),u_{T}(0)\big)^\intercal=(
q_0,q'_0)^\intercal.$$ By appending the equation $q'=p$, the system
\eqref{old-prob} can be turned into
\begin{equation}
\left(
  \begin{array}{c}
    q(t) \\
    p (t) \\
  \end{array}
\right)'+\left(
           \begin{array}{cc}
             0 & -I \\
             M & 0 \\
           \end{array}
         \right)\left(
                  \begin{array}{c}
                    q(t)  \\
                    p(t)  \\
                  \end{array}
                \right)=\left(
                          \begin{array}{c}
                            0 \\
                            f(q(t) ) \\
                          \end{array}
                        \right),\ \ \ \left(
                                   \begin{array}{c}
                                      q(0)  \\
                    p(0)  \\
                                   \end{array}
                                 \right)=\left(
                                           \begin{array}{c}
                                             q_0 \\
                                             q'_0 \\
                                           \end{array}
                                         \right)
                        . \label{eq:ham}%
\end{equation}
We  apply  the  EFCM(k,n) defined by \eqref{methods} to the
first-order differential equations \eqref{eq:ham} and denote the
corresponding numerical solution   by $(v_{E},u_{E})^\intercal$.
From the formulation of EFCMs presented in Section
\ref{sec:Computations}, it follows that $(v_{E},u_{E})^\intercal$ is
the solution of the system
\begin{equation}\begin{aligned}& \left(
                   \begin{array}{c}
                     v_{E}(\xi h) \\
                      u_{E}(\xi h) \\
                   \end{array}
                 \right)
'+\left(
           \begin{array}{cc}
             0 & -I \\
            M & 0 \\
           \end{array}
         \right)\left(
                   \begin{array}{c}
                     v_{E}(\xi h) \\
                      u_{E}(\xi h) \\
                   \end{array}
                 \right)=\left(
                          \begin{array}{c}
                            0 \\
                           \sum\limits_{j=0}^ {n-1}\widehat{P}_j(\xi )\sum\limits_{l=1}^ {k}b_l\widehat{P}_j(c_l)f(v_{E}(c_l h))\\
                          \end{array}
                        \right)
\label{EFCM equation}%
\end{aligned}\end{equation}
with the initial value $$\big( v_{E}(0),u_{E}(0)\big)^\intercal=(
q_0,q'_0)^\intercal.$$

It is obvious that the system \eqref{EFCM equation} as well as the
initial condition is exactly the same as  \eqref{TFCM equation}.
Therefore, we obtain the following theorem.
\begin{theorem}\label{TFCM thm}
The EFCM(k,n) defined by \eqref{methods} reduces to a trigonometric
Fourier collocation method given in \cite{wang-2014} when it is
applied to solve the special first-order differential equations
\eqref{eq:ham},  namely, the second-order oscillatory differential
equations \eqref{old-prob}.
\end{theorem}

\begin{remark}
 It follows from Theorems \ref{HBVM thm}--\ref{TFCM thm} that   EFCMs are an effective extension of
 HBVMs,  Gauss methods, Radau IIA methods and TFCMs. Consequently, EFCMs can
be regarded as a generalization of these existing methods in the
literature.
\end{remark}

\section{Properties of EFCMs} \label{sec:Analysis of the methods}
In this section, we turn to  analysing  the properties of EFCMs,
including their accuracy in preserving the Hamiltonian energy and
the quadratic invariants once the underlying problem is a
Hamiltonian system, their algebraic order {and convergence condition
of the fixed-point iteration.}

 The following result is needed in our analysis, and its proof can be
found in \cite{Brugnano2012}.
\begin{lemma}
\label{Pj lem}Let $f:[0,h]\rightarrow \mathbb{R}^{d}$ have $j$
continuous derivatives in the interval $[0,h]$. Then, we obtain
$\int_{0}^1\widehat{P}_j(\tau)f(\tau h)d\tau=\mathcal{O}(h^{j}).$
\end{lemma}

As a consequence of this lemma, we have
$$\kappa_j(h,v)=\int_{0}^{1}\widehat{P}_j(\tau)g(\tau h,v(\tau
h))d\tau=\mathcal{O}(h^{j}).$$

\subsection{The Hamiltonian case} \label{subsec:Hamiltonian}
Consider  the  following initial-value Hamiltonian
 systems
\begin{equation}
\begin{aligned}
u'(t) =J\nabla H(u(t)),\ \ \ u(0)=u_0 \end{aligned}
\label{H-s}%
\end{equation}
with the Hamiltonian function $H(u)$ and   the  skew-symmetric
matrix $J$. Under the condition that
\begin{equation}J\nabla H(u(t))=g(t,u(t))-Au(t),
\label{condition}%
\end{equation}
the  Hamiltonian
 systems  \eqref{H-s} are identical to the first-order initial value problems of the form
 \eqref{prob}.    The following is an important example
(see, e.g. \cite{Lubich2006,wang-2014}):
\[
H(p,q)=\frac{1}{2}p^{\intercal}p+\dfrac{1}{2}q^{\intercal}Mq+U(q),
\]
where  $M$ is a  symmetric and  positive semi-definite matrix, and
$U$ is a smooth potential with moderately bounded derivatives. This
kind of Hamiltonian system  frequently arises in applied
mathematics, molecular biology, electronics, chemistry, astronomy,
classical mechanics and quantum physics, and it can be expressed by
the following differential equation:
\begin{equation*}
\left(
  \begin{array}{c}
    q \\
    p \\
  \end{array}
\right)'+\left(
           \begin{array}{cc}
             0 & -I \\
             M & 0 \\
           \end{array}
         \right)\left(
                  \begin{array}{c}
                    q \\
                    p \\
                  \end{array}
                \right)=\left(
                          \begin{array}{c}
                            0 \\
                            -\nabla U(q) \\
                          \end{array}
                        \right)
,
\end{equation*}
which is exactly a first-order differential system of the form
\eqref{prob}.

In what follows, we are concerned with the  order of preserving the
Hamiltonian energy when EFCMs are applied to solve the Hamiltonian
 system \eqref{H-s}--\eqref{condition}.

\begin{theorem}\label{energy thm}
Let the quadrature formula in \eqref{methods} {be exact for
polynomials of degree up to $m-1$.} Then, for the  EFCM(k,n) when
applied to the Hamiltonian
 system \eqref{H-s}--\eqref{condition},  we have
$$H(v(h))=H(u_0)+\mathcal{O}(h^{r+1})\ \ \
\textmd{with}\ \  r=\min \{m,2n\}.$$
\end{theorem}
 \begin{proof}
 It follows from \eqref{quadrature truncating
 H-s} and \eqref{condition} that
 \begin{equation*}
\begin{aligned} &H(v(h))-H(u_0)
=h \int_{0}^{1}
\nabla H(v(\xi  h))^{\intercal}v'(\xi  h)d\xi \\
=&h \int_{0}^{1} \nabla H(v(\xi
h))^{\intercal}\Big(\sum\limits_{j=0}^ {n-1}\widehat{P}_j(\xi
)\sum\limits_{l=1}^
{k}b_l\widehat{P}_j(c_l)g(c_l h,v(c_l h))-Av(\xi h)\Big)d\xi \\
=&h \int_{0}^{1}  \Big(g(v(\xi  h))-Av(\xi
h)\Big)^{\intercal}J\Big(\sum\limits_{j=0}^ {n-1}\widehat{P}_j(\xi
)\sum\limits_{l=1}^
{k}b_l\widehat{P}_j(c_l)g(c_l h,v(c_l h))-Av(\xi h)\Big)d\xi \\
=&h \int_{0}^{1}  \Big(g(v(\xi  h))-Av(\xi
h)\Big)^{\intercal}J\Big(g(v(\xi  h))-Av(\xi h)\\
&+\sum\limits_{j=0}^ {n-1}\widehat{P}_j(\xi )\sum\limits_{l=1}^
{k}b_l\widehat{P}_j(c_l)g(c_l h,v(c_l h))-g(v(\xi  h))\Big)d\xi \\
=&h \int_{0}^{1}  \Big(g(v(\xi  h))-Av(\xi
h)\Big)^{\intercal}J\Big(g(v(\xi  h))-Av(\xi h)\Big)d\xi\\
&+h \int_{0}^{1}   \Big(g(v(\xi  h))-Av(\xi
h)\Big)^{\intercal}J\Big(\sum\limits_{j=0}^ {n-1}\widehat{P}_j(\xi
)\sum\limits_{l=1}^ {k}b_l\widehat{P}_j(c_l)g(c_l h,v(c_l
h))-g(v(\xi  h))\Big)d\xi.
\end{aligned}
\end{equation*}
Since J is skew-symmetric, we have $$\int_{0}^{1}  \Big(g(v(\xi
h))-Av(\xi h)\Big)^{\intercal}J\Big(g(v(\xi  h))-Av(\xi
h)\Big)d\xi=0.$$ Thus
\begin{equation*}
\begin{aligned} &H(v(h))-H(u_0)\\
=&h \int_{0}^{1}    \nabla H(v(\xi  h))^T\Big(\sum\limits_{j=0}^
{n-1}\widehat{P}_j(\xi )\sum\limits_{l=1}^
{k}b_l\widehat{P}_j(c_l)g(c_l h,v(c_l
h))-g(v(\xi  h))\Big)d\xi\\
=&h \int_{0}^{1}    \nabla H(v(\xi  h))^T\Big(\sum\limits_{j=0}^
{n-1}\widehat{P}_j(\xi )\sum\limits_{l=1}^
{k}b_l\widehat{P}_j(c_l)g(c_l h,v(c_l h))-\sum\limits_{j=0}^
{+\infty}\widehat{P}_j(\xi
)\kappa_j(h,v)\Big)d\xi\\
=&-h \int_{0}^{1}    \nabla H(v(\xi  h))^T\Big(\sum\limits_{j=0}^
{n-1}\widehat{P}_j(\xi )\Delta_j(h,v)+\sum\limits_{j=n}^
{+\infty}\widehat{P}_j(\xi )\kappa_j(h,v)\Big)d\xi \\
=&-h \sum\limits_{j=0}^ {n-1}\int_{0}^{1}    \nabla H(v(\xi
h))^T\widehat{P}_j(\xi )d\xi\Delta_j(h,v)-h\sum\limits_{j=n}^
{+\infty}\int_{0}^{1}  \nabla H(v(\xi h))^T\widehat{P}_j(\xi
)d\xi\kappa_j(h,v).
\end{aligned}
\end{equation*}
From Lemma \ref{Pj lem},  we have
\begin{equation*}
\begin{aligned} &H(v(h))-H(u_0)
=-h  \sum\limits_{j=0}^ {n-1}\mathcal{O}(h^{j}\times
h^{m-j})-h\sum\limits_{j=n}^{\infty} \mathcal{O}(h^{j}\times h^{j})
=\mathcal{O}(h^{m+1})+\mathcal{O}(h^{2n+1}),
\end{aligned}
\end{equation*}
which shows the  result of the theorem. 
\end{proof}

\subsection{The quadratic invariants} \label{subsec:quadratic}
Quadratic invariants appear often in applications and we thus  pay
attention to the quadratic invariants of \eqref{prob} in this
subsection. Consider the following quadratic function
$$Q(u)=u^{\intercal}Cu$$ with  a symmetric square matrix $C$. It is
an invariant
 of \eqref{prob}  provided $u^{\intercal}C(g(t,u)-Au)=0$   holds.

\begin{theorem}\label{invariant thm}
Let the quadrature formula in \eqref{methods} {be exact for
polynomials of degree up to $m-1$}, then
$$Q(v(h))=Q(u_0)+\mathcal{O}(h^{r+1}) \ \
\textmd{with}\ \    r=\min \{m,2n\}.$$
\end{theorem}
 \begin{proof}
It follows from  the definition of quadratic function $Q$ that
\begin{equation*}
\begin{aligned} &Q(v(h))-Q(u_0)\\
=&\int_{0}^{1}
d Q(v(\xi  h))=\int_{0}^{1}\frac{ d Q(v(\xi  h))}{d\xi}d\xi=2h\int_{0}^{1}v^{\intercal}(\xi  h)Cv'(\xi  h)d\xi\\
=&2h\int_{0}^{1}v^{\intercal}(\xi  h)C\Big(\sum\limits_{j=0}^
{n-1}\widehat{P}_j(\xi )\sum\limits_{l=1}^
{k}b_l\widehat{P}_j(c_l)g(c_l h,v(c_l h))-Av(\xi h)\Big)d\xi\\
=&2h\int_{0}^{1}v^{\intercal}(\xi  h)C\Big(g(\xi h,v(\xi h))-Av(\xi
h)\\&+\sum\limits_{j=0}^ {n-1}\widehat{P}_j(\xi )\sum\limits_{l=1}^
{k}b_l\widehat{P}_j(c_l)g(c_l h,v(c_l h))-g(\xi h,v(\xi
h))\Big)d\xi.
\end{aligned}
\end{equation*}
Since   $u^{\intercal}C(g(t,u)-Au)=0$, we obtain
\begin{equation*}
\begin{aligned} &Q(v(h))-Q(u_0)\\=&
2h\int_{0}^{1}v^{\intercal}(\xi  h)C\Big(\sum\limits_{j=0}^
{n-1}\widehat{P}_j(\xi )\sum\limits_{l=1}^
{k}b_l\widehat{P}_j(c_l)g(c_l h,v(c_l h))-g(\xi h,v(\xi h))\Big)d\xi\\
=&-2h\int_{0}^{1}v^{\intercal}(\xi  h)C\Big(\sum\limits_{j=0}^
{n-1}\widehat{P}_j(\xi )\Delta_j(h,v)+\sum\limits_{j=n}^
{+\infty}\widehat{P}_j(\xi )\kappa_j(h,v)\Big)d\xi\\
=&-2h \sum\limits_{j=0}^ {n-1}\int_{0}^{1}   v^{\intercal}(\xi
h)\widehat{P}_j(\xi )d\xi C\Delta_j(h,v)-2h\sum\limits_{j=n}^
{+\infty}\int_{0}^{1}  v^{\intercal}(\xi  h)\widehat{P}_j(\xi )d\xi C\kappa_j(h,v) \\
=&-2h  \sum\limits_{j=0}^ {n-1}\mathcal{O}(h^{j}\times
h^{m-j})-2h\sum\limits_{j=n}^{\infty} \mathcal{O}(h^{j}\times h^{j})
=\mathcal{O}(h^{m+1})+\mathcal{O}(h^{2n+1}),
\end{aligned}
\end{equation*}
which  proves the theorem.
\end{proof}

\subsection{Algebraic order} \label{subsec:order}
As the importance of different qualitative features, a discussion of
the qualitative theory of the underlying ODEs is given. Therefore,
in this subsection, we  analyse the algebraic order of  EFCMs in
preserving the accuracy of  the solution $u(t)$.

To express the dependence of the solutions of
$$u'(t)=g(t,u(t))-Au(t)$$ on the initial values, we  denote by
$u(\cdot,\tilde{t}, \tilde{u})$ the solution satisfying the initial
condition $u(\tilde{t},\tilde{t}, \tilde{u})=\tilde{u}$ for any
given $\tilde{t}\in[0,h]$ and set
\begin{equation}
\Phi(s,\tilde{t}, \tilde{u})=\frac{\partial u(s,\tilde{t},
\tilde{u})}{\partial \tilde{u}}.
\label{Phi}%
\end{equation}
Recalling the elementary theory of ordinary differential equations,
we have the following standard result   (see, e.g.  \cite{hale1980})
\begin{equation}
\frac{\partial u(s,\tilde{t}, \tilde{u})}{\partial
\tilde{t}}=-\Phi(s,\tilde{t}, \tilde{u})(g(\tilde{t},\tilde{
u})-A\tilde{ u}).
\label{standard result}%
\end{equation}

The following theorem states the result on the  algebraic order of
the novel EFCMs.
\begin{theorem}\label{order thm}
Let the quadrature formula in \eqref{methods} {be exact for
polynomials of degree up to $m-1$}. Then we have
$$ u(h)-v(h)=\mathcal{O}(h^{r+1})\ \ \ \textmd{with}\ \
  r=\min \{m,2n\},$$ for the  EFCM(k,n) defined
by \eqref{methods}.
\end{theorem}
\begin{proof}It follows from Lemma \ref{Pj lem}, \eqref{Phi} and \eqref{standard result}
that
\begin{equation*}
\begin{aligned} & u(h)-v(h)
= u(h,0,  u_0)- u\big(h,h, v(h)\big)=- \int_{0}^{h}
\frac{d u\big(h,\tau, v(\tau)\big)}{d\tau}d\tau\\
=& - \int_{0}^{h}\Big[ \frac{\partial u\big(h,\tau,
v(\tau)\big)}{\partial \tilde{t}}
+\frac{\partial u\big(h,\tau, v(\tau)\big)}{\partial \tilde{ u}}v'(\tau)\Big]d\tau\\
=& h \int_{0}^{1}\Phi\big(h,\xi h, v(\xi h)\big)\Big[g\big(\xi h,v(\xi h)\big)-Av(\xi h)-v'(\xi h)\Big]d\xi \\
=& h \int_{0}^{1}\Phi\big(h,\xi h, v(\xi h)\big)
\Big[\sum\limits_{j=0}^ {+\infty}\widehat{P}_j(\xi
)\kappa_j(h,v)-Av(\xi h)\\
&-\sum\limits_{j=0}^ {n-1}\widehat{P}_j(\xi )\sum\limits_{l=1}^
{k}b_l\widehat{P}_j(c_l)g(c_l h,v(c_l h))+Av(\xi h)\Big]d\xi\\
 =& h \int_{0}^{1}\Phi\big(h,\xi h, v(\xi h)\big)
\Big[\sum\limits_{j=0}^ {+\infty}\widehat{P}_j(\xi
)\kappa_j(h,v)-\sum\limits_{j=0}^ {n-1}\widehat{P}_j(\xi
)\sum\limits_{l=1}^
{k}b_l\widehat{P}_j(c_l)g(c_l h,v(c_l h))\Big]d\xi\\
 =& h \int_{0}^{1}\Phi\big(h,\xi h, v(\xi h)\big)
\Big[\sum\limits_{j=n}^ {+\infty}\widehat{P}_j(\xi
)\kappa_j(h,v)+\sum\limits_{j=0}^ {n-1}\widehat{P}_j(\xi
)\Delta_j(h,v)\Big]d\xi\\
 =& h \sum\limits_{j=n}^ {+\infty}\int_{0}^{1}\Phi\big(h,\xi h, v(\xi h)\big)
 \widehat{P}_j(\xi )d\xi\kappa_j(h,v)+h\sum\limits_{j=0}^
{n-1}\int_{0}^{1}\Phi\big(h,\xi h, v(\xi h)\big)\widehat{P}_j(\xi
) d\xi\Delta_j(h,v)\\
=&h \Big(\sum\limits_{j=n}^{\infty} \mathcal{O}(h^{j} \times
h^{j})+\sum\limits_{j=0}^ {n-1}\mathcal{O}(h^{j}\times
h^{m-j})\Big)=\mathcal{O}(h^{2n+1})+\mathcal{O}(h^{m+1})
\\=& \mathcal{O}(h^{\min \{m,2n\}+1}).
\end{aligned}
\end{equation*}
The proof is complete. 
\end{proof}

\begin{remark}
 This result means that  choosing a suitable   quadrature formula as well as   a suitable value of $n$ in
\eqref{methods} can yield an  EFCM   of arbitrarily high order. This
manipulation is very simple and convenient, and it opens up  a new
possibility to construct higher--order EFCMs in a simple and routine
manner.
\end{remark}

\begin{remark}It is well known that rth-order numerical methods can preserve the
Hamiltonian energy or the quadratic invariant with at least rth
degree of accuracy, but unfortunately it follows from the analysis
of Subsections \ref{subsec:Hamiltonian} and \ref{subsec:quadratic}
that our methods preserve the Hamiltonian energy and the quadratic
invariant with only rth degree of accuracy.
\end{remark}

\subsection{Convergence condition of the fixed-point
iteration} \label{subsec:Convergence} It is worthy noting that
usually the EFCM(k,n) defined by \eqref{methods} constitutes of a
system of implicit equations for the determination of $v_i$, and the
iterative computation is required. {In this paper, we only consider
using the fixed-point iteration  in practical computation. Other
iteration methods such as waveform relaxation methods, Krylov
subspace methods and preconditioning will be analysed in a future
research. } For the convergence of the fixed-point iteration for the
EFCM(k,n)
  \eqref{methods}, we have the following result.
\begin{theorem}
\label{convergence}  Assume that   $g$ satisfies a Lipschitz
condition in the variable $u$, i.e. there exists a constant $L$ with
the property:
$$\norm{g(t,u_1)-g(t,u_2)}\leq L\norm{u_1-u_2}.$$ If
\begin{equation}0<h<\dfrac{1}{L\dfrac{Cr^2(e^{\omega}-1)}{\omega}\max\limits_{i,j=
1,\cdots, k} c_i |b_j|}, \label{Convergence condition}
\end{equation} then, the fixed-point  iteration for
the EFCM(k,n)   \eqref{methods} is convergent. Here,   $C$ and
$\omega$ are constants independent of $A$. For a  quadrature
formula, generally speaking, not all of the node points $c_i \
(i=1,2,\ldots,k)$ are equal to zero, and this ensures that
$\max\limits_{i,j= 1,\cdots, k} c_i |b_j|\neq0.$
\end{theorem}
\begin{proof}
Following Definition \ref{numerical method}, the first formula
  of   \eqref{methods} can be rewritten as
\begin{align}
Q  & =e^{-cV}u_{0}+hA(V)g(ch,Q),
\end{align}
where $c=(c_1,c_2,\ldots,c_k)^\intercal,\
Q=(v_1,v_2,\ldots,v_k)^{\intercal},\ A(V)=(a_{ij}(V))_{k\times k}$
and $a_{ij}(V)$ are defined as
$$a_{ij}(V):=c_ib_j\sum\limits_{l=0}^ {n-1}I_{l,c_i}(V)
\widehat{P}_l(c_j).
$$
It follows from \eqref{Pj} that $|\widehat{P}_j|\leq\sqrt{2j+1}.$ We
then obtain \begin{equation*}
\begin{aligned}\norm{a_{ij}(V)}&\leq c_i|b_j|\sum\limits_{l=0}^ {n-1}\sqrt{2l+1}\int_{0}^1|\widehat{P}_l(c_iz)|\norm{e^{-(1
-z)c_iV}}dz\\
&\leq c_i|b_j|\sum\limits_{l=0}^ {n-1}(2l+1)\int_{0}^1 \norm{e^{-(1
-z)c_iV}}dz.
\end{aligned}
\end{equation*}
Furthermore, we get
 \begin{equation*}
\begin{aligned}\norm{a_{ij}(V)}
&\leq c_i|b_j|\sum\limits_{l=0}^ {n-1}(2l+1)C\int_{0}^1 e^{\omega(1
-z)}dz=Cc_i
|b_j|r^2(e^{\omega}-1)/\omega,\\
\end{aligned}
\end{equation*}
 which yields
$$\norm{A(V)}\leq \frac{Cr^2(e^{\omega}-1)}{\omega}\max\limits_{i,j=
1,\cdots, k} c_i |b_j|.$$
 Let
$$\varphi(x)=e^{-cV}u_{0}+hA(V)g(ch,x).$$ Then  we have
\begin{equation*}
\begin{aligned}
\norm{\varphi(x)-\varphi(y)} &=\norm{hA(V)g(ch,x)-hA(V)g(ch,y)} \leq
 hL\norm{A(V)} \norm{x-y}\\
 &\leq hL\frac{Cr^2(e^{\omega}-1)}{\omega}\max\limits_{i,j=
1,\cdots, k} c_i |b_j|\norm{x-y},
\end{aligned}
\end{equation*}
which shows that $\varphi(x)$ is a contraction under the assumption
\eqref{Convergence condition}. The well-known Contraction Mapping
Theorem then ensures the convergence of the fixed-point
iteration.
\end{proof}

In what follows, we discuss the convergence of the fixed-point
iteration for the HBVM(k,n) \eqref{methods0} for solving
  \eqref{prob}. When the HBVM(k,n) \eqref{methods0}  is applied to
solve
$$u^{\prime}(t)=g(t,u(t))-Au(t), \ \ \ u(0)=u_0,$$
the scheme of HBVM(k,n) becomes
\begin{equation}\label{HBVM}
\begin{aligned} v_i=&u_0+c_ih\sum\limits_{l=1}^ {k}b_l\Big(\sum\limits_{j=0}^ {n-1} \tilde{I}_{j,c_i}
\widehat{P}_j(c_l)\Big)(g(c_lh,v_l)-Av_l),\ \
i=1,2,\ldots,k,\\
v(h)&=u_0+h \sum\limits_{l=1}^ {k}b_l(g(c_lh,v_l)-Av_l).
\end{aligned}
\end{equation}
The first formula of  \eqref{HBVM} is also implicit and it requires
the iterative computation as well. Under the assumption that $g$
satisfies a Lipschitz condition in the variable $u$, in order to
analyse the convergence for the fixed-point iteration for the
formula \eqref{HBVM},  we denote  the iterative function by
$$\psi(x)=u_{0}+h\tilde{A}(g(ch,x)-Ax),$$
where $\tilde{A}=(\tilde{a}_{ij})_{k\times k}$ and
$\tilde{a}_{ij}=c_ib_j\sum\limits_{l=0}^ {n-1}\tilde{I}_{l,c_i}
\widehat{P}_l(c_j).$

Then, we have
\begin{equation*}
\begin{aligned}
\norm{\psi(x)-\psi(y)}
&=\norm{h\tilde{A}(g(ch,x)-Ax)-h\tilde{A}(g(ch,y)-Ay)}\\
 & \leq
 hL\norm{\tilde{A}} \norm{x-y}+h\norm{\tilde{A}}\norm{A} \norm{x-y}\\
 &\leq h(L+\norm{A})\max\limits_{i,j= 1,\cdots,
k}|\tilde{a}_{ij}| \norm{x-y},
\end{aligned}
\end{equation*}
which means that if
\begin{equation*}0<h<\dfrac{1}{(L+\norm{A})\max\limits_{i,j= 1,\cdots,
k}|\tilde{a}_{ij}|},
\end{equation*} then, the fixed-point  iteration for
the HBVM(k,n) is convergent.

\begin{remark} It is very clear that  the convergence of HBVM(k,n)
when applied to $u^{\prime}(t)=g(u(t))-Au(t)$ depends on $\norm{A},$
and the larger $\norm{A}$ becomes, the smaller the stepsize is
required. Whereas, it is of prime importance  to note that from
\eqref{Convergence condition}, the convergence of EFCM(k,n) is
independent of $\norm{A}$. { This fact implies that EFCMs have the
better  convergence condition than HBVMs, especially when $\norm{A}$
is large, such as when the problem \eqref{prob} is a stiff system.
This point will be numerically demonstrated by the
  experiments carried out in next section.} We also note
that an efficient implementation of HBVMs
 has been considered in  \cite{Brugnano2011} and this technique is suitable for stiff first--order  and second--order problems.
\end{remark}

 \section{A practical EFCM and numerical experiments}
 \label{numerical experiments} As an  illustrative example of  EFCMs, we choose the 2-point  Gauss--Legendre
 quadrature as the quadrature formula in \eqref{methods}, that is exact for all polynomials of degree $\leq
 3$. This means that $k=2$ in the $k$-point Gauss--Legendre
 quadrature and this case gives
\begin{equation}\begin{aligned}
&c_1=\frac{3-\sqrt{3}}{6},\ \
c_2=\frac{3+\sqrt{3}}{6},\\
&b_1=\frac{1}{2},\qquad \ \ \ \ b_2=\frac{1}{2}.\\
\label{GL}
\end{aligned}\end{equation}  Then we choose $n=2$
in \eqref{methods} and denote the corresponding exponential Fourier
collocation method as EFCM(2,2). {After some calculations, the
scheme of this method can be expressed by
\begin{equation}
\begin{aligned} v_1=&\varphi_{0}(-c_1V)u_0+
\frac{h}{6}\Big(\sqrt{3}\varphi_{1}(-c_1V)+(3-2\sqrt{3})\varphi_{2}(-c_1V)\Big)g(c_1
h,v_1)\\
&+\frac{3-2\sqrt{3}}{6}h\Big(\varphi_{1}(-c_1V)-\varphi_{2}(-c_1V)\Big)g(c_2
h,v_2),\\
 v_2=&\varphi_{0}(-c_2V)u_0+\frac{3+2\sqrt{3}}{6}h\Big(\varphi_{1}(-c_2V)-\varphi_{2}(-c_2V)\Big)
g(c_1h,v_1)\\
&+\frac{h}{6}\Big(-\sqrt{3}\varphi_{1}(-c_2V)+(3+2\sqrt{3})\varphi_{2}(-c_2V)\Big)g(c_2
h,v_2),\\
 v(h)&=\varphi_{0}(-V)u_0+\frac{h}{2}\Big((1+\sqrt{3})\varphi_{1}(-V)-2\sqrt{3}\varphi_{2}(-V)\Big)g(c_1 h,v
 _1)\\
 &+\frac{h}{2}\Big((1-\sqrt{3})\varphi_{1}(-V)+2\sqrt{3}\varphi_{2}(-V)\Big)g(c_2 h,v
 _2).
\end{aligned}
\label{method EFCM(2,2)}%
\end{equation}
}  When $A\rightarrow 0$, the method EFCM(2,2) reduces to HBVM(2,2)
given in \cite{Brugnano2012}, which coincides with the two-stage
Gauss method given  in \cite{hairer2006}. Various examples of EFCMs
can be obtained by choosing different quadrature formula and
different values of
 $n$, and we do not go further on this point in this paper for
brevity.

%
%
%

In order to   show the efficiency and robustness of the fourth order
method EFCM(2,2), the integrators we select for comparisons are also
of order four and we denote them as follows:

\begin{itemize}\itemsep=-0.2mm
\item EFCM(2,2):  the   EFCM(2,2) method of order four derived in this section;

\item HBVM(2,2):  the Hamiltonian Boundary Value
Method of order four in \cite{Brugnano2012} which coincides with the
two-stage Gauss  method in \cite{hairer2006};

\item {EPCM5s4:  the  fourth-order energy-preserving collocation
method (the case $s=2$)  in \cite{hairer2010} with the integrals
approximated by the Lobatto quadrature of order eight, which is
 precisely the ``extended Labatto
IIIA method{\index{extended Labatto IIIA method}} of order four" in
\cite{Iavernaro2009};}

\item EERK5s4:   the  explicit five-stage exponential Runge--Kutta method of order four
derived in \cite{Hochbruck2005}.

\end{itemize}

 It is noted  that   the first three  methods  are   implicit and  we
 use {one fixed-point iteration in the practical computations for showing  the
work precision diagram (the gloal error  versus the execution time)
as well as energy conservation for a Hamiltonian system. For each
problem,  we also present the requisite total numbers of iterations
for implicit methods when choosing different error tolerances in the
fixed-point iteration.}
In all the numerical experiments, the
 matrix exponential is calculated  by  the algorithm given in  \cite{Higham2009}.

 \textbf{Problem 1.} We first consider the H\'{e}non-Heiles Model which is
created for describing stellar motion (see, e.g.
\cite{Brugnano-SIAM,hairer2006}). The Hamiltonian function of the
system is given by
\[
H(p,q)=\dfrac{1}{2}(p_{1}^{2}+p_{2}^{2})+\dfrac{1}{2}(q_{1}^{2}+q_{2}%
^{2})+q_{1}^{2}q_{2}-\dfrac{1}{3}q_{2}^{3}.
\]
This is  identical to  the following first-order differential
equations
\[
\begin{aligned}
 \left(
   \begin{array}{c}
     q_{1} \\
     q_{2} \\
     p_{1} \\
     p_{2} \\
   \end{array}
 \right)'+\left(
           \begin{array}{cccc}
             0 & 0 & -1 & 0 \\
             0 & 0 & 0 & -1 \\
             1 & 0 & 0& 0 \\
             0 & 1 & 0 & 0 \\
           \end{array}
         \right)
  \left(
   \begin{array}{c}
     q_{1} \\
     q_{2} \\
     p_{1} \\
     p_{2} \\
   \end{array}
 \right)=\left(
   \begin{array}{c}
    0 \\
     0 \\
    -2q_1q_2 \\
    -q_1^2+q_2^2 \\
   \end{array}
 \right).\
\end{aligned}
\]
The initial values are chosen as
$$\big(q_1(0),q_2(0),p_1(0),p_2(0)\big)^{\intercal}=\Big(\sqrt{\dfrac{11}{96}},0,0,\dfrac{1}{4}\Big)^{\intercal}.$$
It is noted that we use the result of the standard ODE45  in MATLAB
as the true solution for this problem and the next problem.
 We first solve the problem in the interval
$[0, 1000]$ with different stepsizes {$h=1/2^{i},\ i =2, 3,4,5$. The
work-precision diagram is}  presented in Figure \ref{fig1-1} (i).
Then, we integrate this problem with the stepsize $h=1.5$ in the
interval $[0, 3000].$ See Figure \ref{fig1-1} (ii) for the energy
conservation for different methods. { We also solve the problem in
$[0, 10]$ with  $h=0.01$ by the three implicit methods and display
the total numbers of iterations in Table \ref{pro1 tab} for
different error tolerances (tol) chosen in the fixed-point
iteration.}

\begin{figure}[ptb]
\centering\tabcolsep=2mm
\begin{tabular}
[c]{cccc}%
\includegraphics[width=5cm,height=6cm]{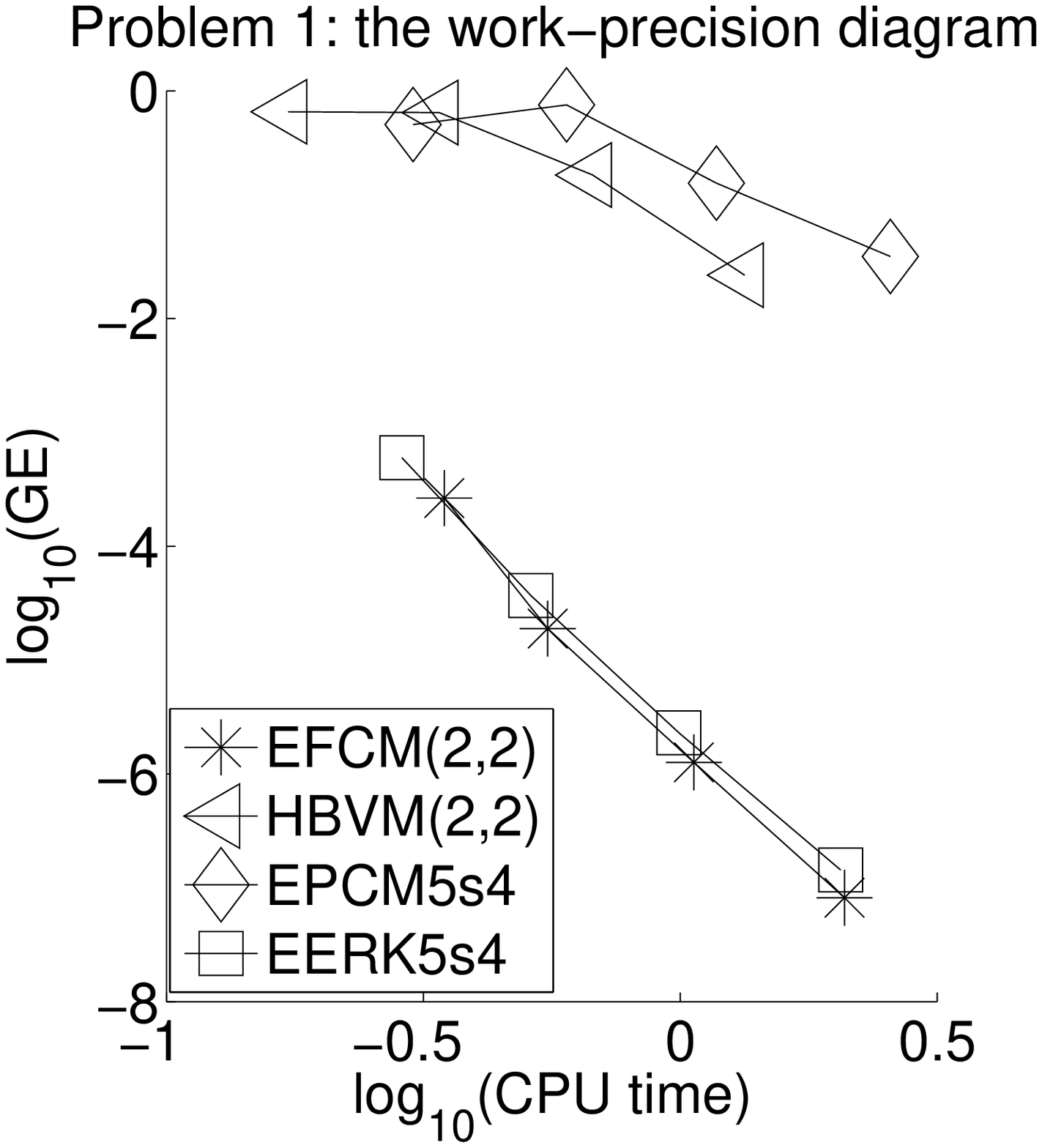} &
\includegraphics[width=5cm,height=6cm]{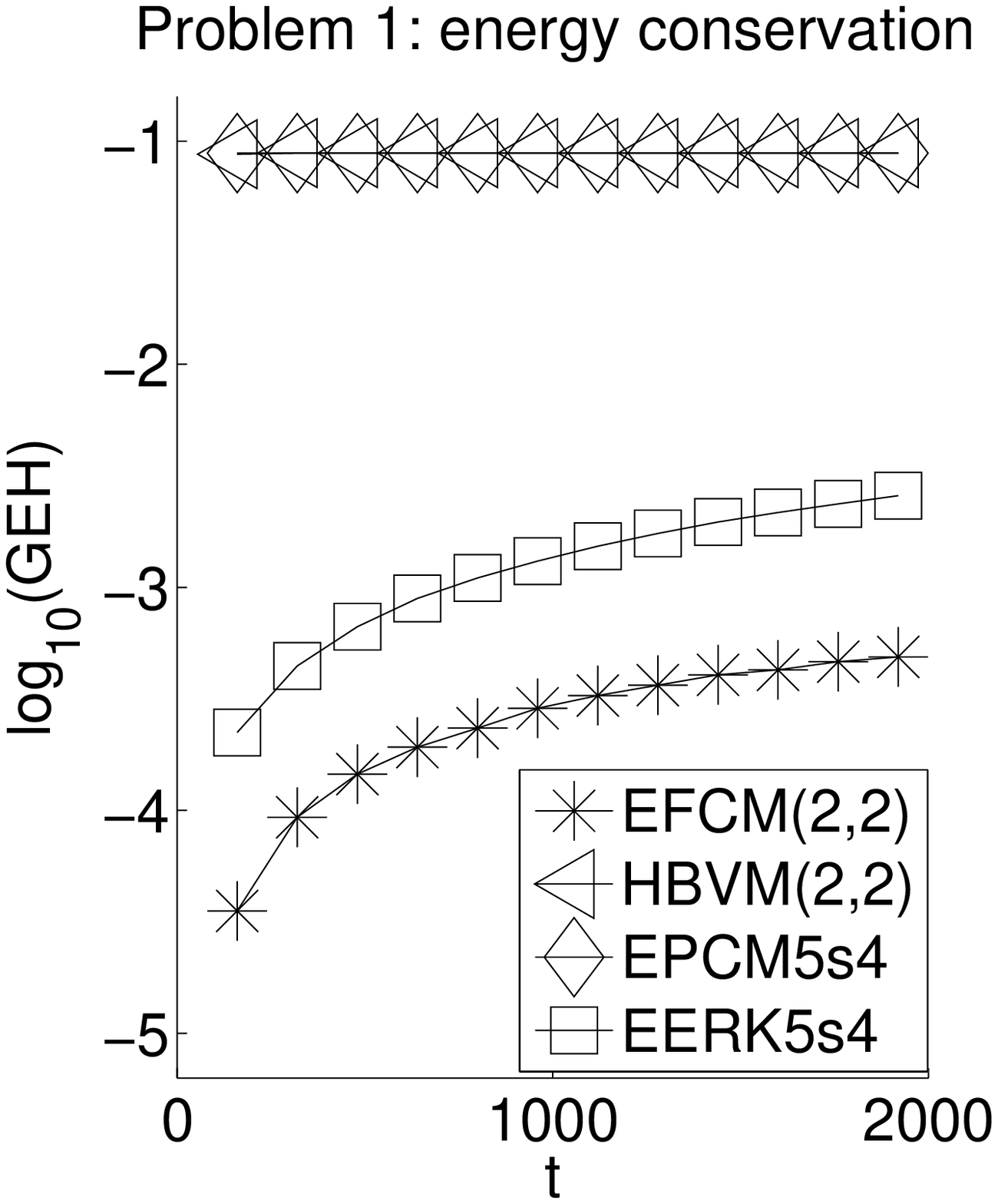} &\\
{\small (i)} & {\small (ii)}
\end{tabular}
\caption{Results for Problem 1. (i): The $\log$-$\log$ plot of the
maximum global error ($GE$) over the integration interval against
{the execution time.} (ii):\ The logarithm of the global error of
Hamiltonian
$GEH=|H_{n}-H_{0}|$ against $t$.}%
\label{fig1-1}%
\end{figure}

\begin{table}$$
\begin{array}{|c|c|c|c|c|c|}
\hline
\text{Methods} &tol=1.0e-006  &tol=1.0e-008   &tol=1.0e-010 &tol=1.0e-012    \\
\hline
\text{EFCM(2,2)} &  2000 &   2000&  2000 &  3000
 \cr
\text{HBVM(2,2)}  &2000 & 3000 &  3769 & 4000\cr
\text{EPCM5s4}   &2000 & 3000 &  4000 & 4999\cr
 \hline
\end{array}
$$
\caption{Results for Problem 1. The total  numbers of iterations for
different error tolerances (tol).} \label{pro1 tab}
\end{table}

\textbf{Problem 2.} The Fermi--Pasta--Ulam problem is an important
model for simulations in statistical mechanics which is considered
in \cite{Lubich2006,hairer2000,hairer2006,wu2012-3,wu-2012-BIT}. It
is a Hamiltonian system with the Hamiltonian
\[%
\begin{aligned}%
H(x,y) &
=\dfrac{1}{2}\textstyle\sum\limits_{i=1}^{2m}y_{i}^{2}+\dfrac
{\omega^{2}}{2}\textstyle\sum\limits_{i=1}^{m}x_{m+i}^{2}+\dfrac{1}{4}%
\Big[(x_{1}-x_{m+1})^{4}\\
& +\textstyle\sum\limits_{i=1}^{m-1}(x_{i+1}-x_{m+i-1}-x_{i}-x_{m+i}%
)^{4}+(x_{m}+x_{2m})^{4}\Big].
\end{aligned}
\]
 This results in
\begin{equation}
\begin{aligned} &\left(
                   \begin{array}{c}
                     x \\
                     y \\
                   \end{array}
                 \right)'
+\left(\begin{array}{cc}
\mathbf{0}_{2m\times 2m} &-I_{2m}\\
M&\mathbf{0}_{2m\times 2m}
\end{array}\right)
 \left(
                   \begin{array}{c}
                     x \\
                     y \\
                   \end{array}
                 \right)
=\left(
   \begin{array}{c}
     0 \\
    -\nabla U(x) \\
   \end{array}
 \right)
, \qquad t\in[0,t_{\textmd{end}}],\end{aligned}
\label{first order-dim-prob}%
\end{equation}
where
\begin{gather*}
M=\left(
\begin{array}
[c]{cc}%
\mathbf{0}_{m\times m} & \mathbf{0}_{m\times m}\\
\mathbf{0}_{m\times m} & \omega^{2}I_{m\times m}%
\end{array}
\right)  ,\\
U(x)=\dfrac{1}{4}\Big[(x_{1}-x_{m+1})^{4}+\textstyle\sum\limits_{i=1}%
^{m-1}(x_{i+1}-x_{m+i-1}-x_{i}-x_{m+i})^{4}+(x_{m}+x_{2m})^{4}\Big].
\end{gather*}
We choose
\[
m=3,\ \omega=50,\ x_{1}(0)=1,\ y_{1}(0)=1,\
x_{4}(0)=\dfrac{1}{\omega},\ y_{4}(0)=1,
\]
and choose zero for the remaining initial values.   The system is
integrated in the interval $[0,10]$ with {the stepsizes $h= 1/2^k,\
k=3,4,5,6.$ We plot the work-precision diagram}   in Figure
\ref{fig2} (i).
 Then, we solve this problem in the interval $[0,1000]$
 with the stepsize {$h=1/10$} and present the energy conservation in Figure \ref{fig2} (ii).
  Here, it is noted that we do not plot some points in  Figure \ref{fig2}   when
  the errors of the
  corresponding numerical results are too large. Similar situation occurs in the next
  two problems. Furthermore, we solve the problem in   $[0, 10]$ with
$h=0.01$ to show the  convergence rate of iterations for the three
implicit methods. Table \ref{pro2 tab} lists the total numbers of
iterations for different error tolerances.
\begin{figure}[ptb]
\centering\tabcolsep=2mm
\begin{tabular}
[c]{cccc}%
\includegraphics[width=5cm,height=6cm]{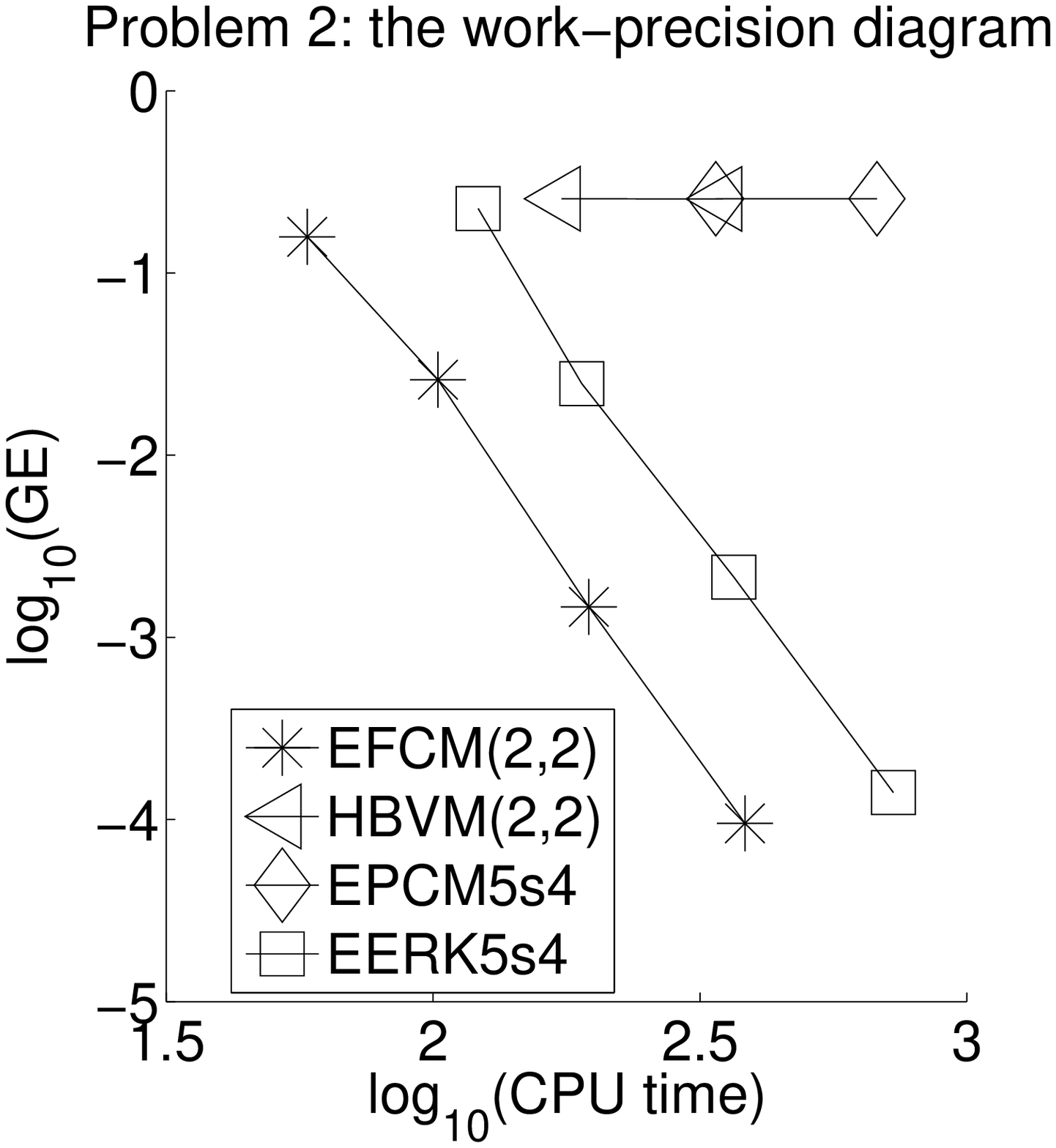} &
\includegraphics[width=5cm,height=6cm]{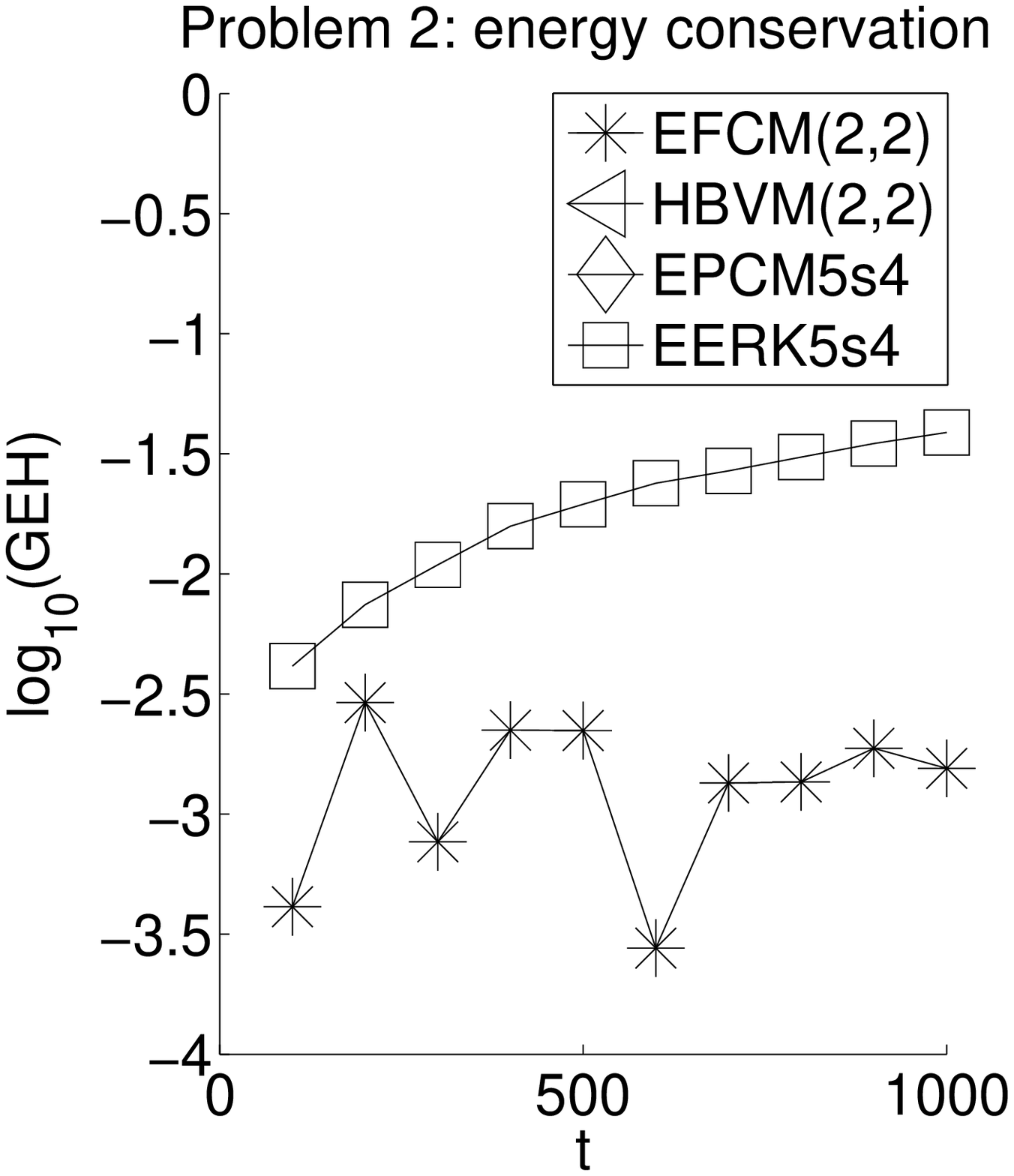} &\\
{\small (i)} & {\small (ii)}
\end{tabular}
\caption{Results for Problem 2. (i): The $\log$-$\log$ plot of the
maximum global error ($GE$) over the integration interval  against
{the execution time.} (ii):\ The logarithm of the global error of
Hamiltonian
$GEH=|H_{n}-H_{0}|$ against $t$.}%
\label{fig2}%
\end{figure}
\begin{table}$$
\begin{array}{|c|c|c|c|c|c|}
\hline
\text{Methods} &tol=1.0e-006  &tol=1.0e-008   &tol=1.0e-010 &tol=1.0e-012    \\
\hline
\text{EFCM(2,2)} &2000&2080&2998&3027 \cr
\text{HBVM(2,2)}  & 6801     &   9291    &   10980    &   13912\cr
\text{EPCM5s4}   &   9937    &    11925  &      14844   &     16945\cr
 \hline
\end{array}
$$
\caption{Results for Problem 2. The total  numbers of iterations for
different error tolerances (tol).} \label{pro2 tab}
\end{table}

\textbf{Problem 3.} Consider the semilinear parabolic problem (this
problem has been considered in \cite{Hochbruck2005})
\[
\frac{\partial u}{\partial t}(x,t)=\dfrac{\partial^{2}u}{\partial
x^{2}}(x,t)+\frac{1}{1+u(x,t)^2}+\Phi(x,t)
\]
for $x\in[0,1]$ and $t\in[0,1],$ subject to homogeneous Dirichlet
boundary conditions. The source function $\Phi(x,t)$ is chosen in
such a way that the exact solution of the problem is
$u(x,t)=x(1-x)\textmd{e}^t$.

We discretise this problem in space by using second-order symmetric
differences with 1000 grid points. The problem is solved in the
interval $[0, 1]$ with different stepsizes $h=1/2^{i},\ i =2,3,4,5.$
The work-precision diagram is  presented in Figure
\ref{fig3-NEW}. 
Then, the problem is solved in   $[0, 1]$ with $h=\frac{1}{10}$ to
show the convergence rate of iterations. See Table \ref{pro3-NEW
tab} for the total numbers of iterations for different error
tolerances.

\begin{figure}[ptb] \centering\tabcolsep=2mm
\begin{tabular}
[c]{cccc}%
\includegraphics[width=5cm,height=6cm]{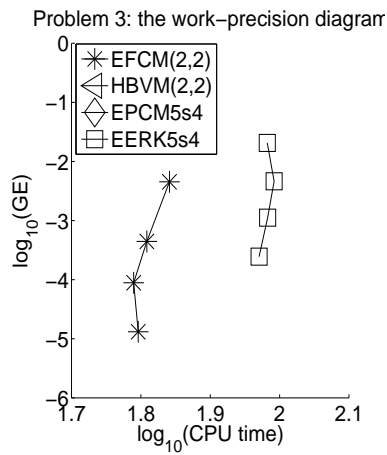}
\end{tabular}
\caption{Results for Problem 3. The $\log$-$\log$ plot of the
maximum global error ($GE$) over the integration
interval against  {the execution time.}}%
\label{fig3-NEW}%
\end{figure}

\begin{table}$$
\begin{array}{|c|c|c|c|c|c|}
\hline
\text{Methods} &tol=1.0e-006  &tol=1.0e-008   &tol=1.0e-010 &tol=1.0e-012    \\
\hline
\text{EFCM(2,2)} &40   & 50  &  60  &  73 \cr
\text{HBVM(2,2)}  &    86 &   86   & 86  &  86\cr
\text{EPCM5s4}  &  87  &  87 &   87  &  87\cr
 \hline
\end{array}
$$
\caption{Results for Problem 3. The total  numbers of iterations for
different error tolerances (tol).} \label{pro3-NEW tab}
\end{table}

From the results, it can be clearly observed   that the novel method
EFCM(2,2) provides a considerably more accurate numerical solution
than other methods and preserves well the Hamiltonian energy when
solving Hamiltonian  systems. Moreover, our method EFCM(2,2)
requires less fixed-point iterations than both HBVM(2,2) and
EPCM5s4, which is important in long-term computations.

\section{Conclusions} \label{sec:conclusions}

In this paper, we  formulated and analysed the novel   methods EFCMs
for solving systems of first-order differential
 equations. The novel EFCMs  are  an efficient  kind of exponential integrators, and their construction takes full advantage of
   the variation-of-constants formula, the local Fourier expansion
and  collocation methods.  We discussed the connections with HBVMs,
Gauss methods,   Radau IIA methods and TFCMs.  It turned out that
 the first three traditional methods can be attained
by letting $A\rightarrow 0$ in the corresponding EFCMs, and applying
EFCMs to the second-order oscillatory differential equation
\eqref{old-prob} yields TFCMs. The properties of EFCMs were also
analysed and it was shown that the new EFCMs can reach arbitrarily
high order in a very convenient and simple way.  A practical scheme
of EFCMs was constructed in this paper. The numerical experiments
were carried out and the results affirmatively demonstrate
 that the novel EFCMs have excellent numerical behaviour in
comparison with some existing effective methods in the scientific
literature.

This is a preliminary  research on  EFCMs for first-order ordinary
differential equations  and the authors are  clearly aware that
there are still  some issues which will  be further considered:
\begin{itemize}\itemsep=-0.2mm

\item {The error bounds and
convergence properties of EFCMs for linear and semilinear
 problems will be discussed in another work. }

\item  For the EFCM(k,n) \eqref{methods},   it is assumed that
$k\geq n$ in this paper.   EFCMs with   $k<n$  will be discussed and
 this case maybe not affect  the computational cost associated with the
implementation of the methods for some special systems.   Some
equations and unknowns in the methods may be  removed and we will
consider the efficient implementation of  the novel EFCMs  in a
future research.

\item We only consider the fixed-point
iteration for the EFCMs in this paper. Other iteration methods such
as waveform relaxation methods, Krylov subspace methods and
preconditioning as well as their  actual implementation for EFCMs
will be analysed in future.

\item The shifted Legendre polynomials are  chosen as  an orthonormal
basis to give the Fourier expansion of the function $g(t,u(t))$. We
observe that a different choice of the orthonormal basis would
modify the arguments presented in this paper.  The scheme of the
numerical methods as well as their analysis is then  modified
accordingly.
 Different choices
of the orthonormal basis will be considered in future
investigations.

\item Another issue for future exploration is the application  of our   methodology in   other differential equations
such as   Sch\"{o}rdinger equations and other stiff PDEs.

\end{itemize}

\noindent {\bf Acknowledgments.} Bin Wang was supported  by National
Natural Science Foundation of China (Grant No. 11401333), by Natural
Science Foundation of Shandong Province (Grant No. ZR2014AQ003) and
by China Postdoctoral Science Foundation (Grant No. 2015M580578).
Xinyuan Wu was supported by  National Natural Science Foundation of
China (Grant No. 11271186), by NSFC and RS International Exchanges
Project (Grant No. 113111162), by Specialized Research Foundation
for the Doctoral Program of Higher Education (Grant No.
20130091110041), by 985 Project at Nanjing University (Grant No.
9112020301), by A Project Funded by the Priority Academic Program
Development of Jiangsu Higher Education Institutions. Fanwei Meng
was supported  by National Natural Science Foundation of China
(Grant No. 11171178).  Yonglei Fang was partially supported by
National Natural Science Foundation of China (Grant No. 11571302)
and the foundation of Scientific Research Project of Shandong
Universities (Grant No. J14LI04).

\end{document}